\documentclass[10pt]{amsart}
\usepackage{amsmath}
\usepackage{amssymb}
\usepackage{amsmath}
\usepackage{amsfonts}
\usepackage{latexsym}
\usepackage{color}
\usepackage{graphicx}

\newcommand{\be}{\begin{equation}}
\newcommand{\en}{\end{equation}}
\newcommand{\bea}{\begin{eqnarray}}
\newcommand{\ena}{\end{eqnarray}}
\newcommand{\beano}{\begin{eqnarray*}}
\newcommand{\enano}{\end{eqnarray*}}
\newcommand{\bee}{\begin{enumerate}}
\newcommand{\ene}{\end{enumerate}}

\newcommand{\mc}{\mathcal}
\newcommand{\mb}{\mathbb}

\newcommand{\F}{{\mathbb F}}

\newcommand{\D}{{\mc D}}

\newcommand{\E}{{\mc E}}

\newcommand{\BB}{{\mathfrak B}}

\newcommand{\jj}[1]{[#1]}

\newtheorem{defn}{Definition}[section]
\newtheorem{thm}[defn]{Theorem}
\newtheorem{prop}[defn]{Proposition}
\newtheorem{lemma}[defn]{Lemma}

\newtheorem{example}[defn]{Example}
\newtheorem{rem}[defn]{Remark}

\def\x{\relax\ifmmode {\mbox{*}}\else*\fi}
\newcommand{\beex}{\begin{example}$\!\!${\bf }$\;$\rm }
\newcommand{\enex}{ \end{example}}
\newcommand{\berem}{\begin{rem}$\!\!${\bf }$\;$\rm }
\newcommand{\enrem}{ \end{rem}}
\newcommand{\bedefi}{\begin{defn}$\!\!${\bf }$\;$\rm }
\newcommand{\findefi}{\end{defn}}
\def\H{{\mathcal H}}
\def\S{{\mathcal S}}
\def\L{{\mathcal L}}

\newcommand{\btheta}{\mathbf{\Theta}}

\newcommand{\ip}[2]{\langle {#1}|{#2}\rangle}
\newcommand{\gl}{{\mathfrak L}}
\newcommand{\LDD}{\gl(\D,\D^\times)}

\newcommand{\LBDD}[1]{\mbox{${\sf L}_{\textsf{B}}^{#1}(\D,\D^\times)$}}

\begin{document}
\title[Rigged Hilbert spaces ]
{Rigged Hilbert spaces and contractive families of Hilbert spaces}

\author{Giorgia Bellomonte and Camillo Trapani}

\address{Dipartimento di Matematica ed Applicazioni,
Universit\`a di Palermo, Via Archirafi 34, I-90123 Palermo, Italy}
\email{bellomonte@math.unipa.it; trapani@unipa.it}
\subjclass[2000]{47A70, 46A13, 46M40}

\keywords{Rigged Hilbert spaces, inductive and projective limits}

\begin{abstract}
The existence of a rigged Hilbert space whose extreme spaces are, respectively,  the projective and the inductive limit
 of a directed contractive family of Hilbert spaces is investigated.
It is proved that, when it exists, this rigged Hilbert space is the same as the canonical rigged Hilbert space associated
to a family of closable operators in the central Hilbert space.
\end{abstract}

\maketitle

\section{New introduction}

Indexed families of Hilbert spaces are very common in functional analysis (think, for instance, of Sobolev spaces $W^{k,2}({\mb R}$) or weighted $L^2$-spaces\,) and, if the index set satisfies certain axioms, they can originate several {\em global} structures, like lattices of Hilbert spaces (LHS), nested Hilbert spaces (NHS)\cite{gross}, partial inner product spaces ({\sc Pip}-spaces, \cite{antgross1, antgross2}), see the recent monograph \cite{at_pip} for precise definitions and full analysis.  The spirit of all these constructions (which all present the common feature of going {\em beyond} Hilbert space) is that, in many occasions, it is not the single space that has a relevance, but the whole family; this in particular happens when dealing with operators which can be very singular as regarded on a single space but may behave very regularly if considered (by extensions or restrictions, according to the cases) as operators acting on the family. Just to mention a very simple situation, a closed unbounded operator $A$ with domain $D(A)$ in Hilbert space $\H$ always gives rise to a scale of Hilbert spaces $\H_A \subset \H \subset \H_A^\times$, where $\H_A$ is the Hilbert space obtained giving $D(A)$ the graph norm, and $\H_A^\times$ is its conjugate dual. The operator $A$, which is unbounded in $\H$, is certainly bounded when regarded as a linear map from $\H_A$ into $\H_A^\times$.

The triplet of Hilbert spaces $\H_A \subset \H \subset \H_A^\times$ is a particular example of a {\em rigged Hilbert space} or, as it is also
called, {\em Gel'fand triplet}. Rigged Hilbert spaces and operators acting on them will be the main subject of this paper and so we recall the basic definitions.
 
A {rigged Hilbert space} (RHS) consists of a triplet $(\D
,\H, \D^\times)$ where $\D$ is a dense subspace of $\H$ endowed with
a topology $\sf t$, finer than that induced by the Hilbert norm of
$\H$, and $\D^\times$ is the conjugate dual of $\D[{\sf t}]$. The
space $\D^\times$ is usually endowed with the strong topology ${\sf
t}^\times:= \beta(\D^\times, \D)$
(we write $\D \hookrightarrow \H \hookrightarrow \D^\times$, where $\hookrightarrow$ denotes a continuous embedding with dense range).
 This structure is very familiar in distribution theory (think of the triplet $({\mc S}({\mb R}), L^2(\mb R), {\mc S}'({\mb R}))$ constituted by the Schwartz space of rapidly decreasing $C^\infty$-functions, the Lebesgue $L^2$-space on the real line and  the space of tempered distributions) and RHSs were also considered from the very beginning of the studies on unbounded operator algebras \cite{FL1, FL2, vEinKru}, revealing themselves as a powerful tool in this area \cite{schmud} and for applications to Quantum theories (see, e.g. \cite{at_pip} and references therein).
 RHSs are the simplest example of a {\sc Pip}-space. 
  
  The space $\LDD$ of all continuous linear maps from $\D[\sf t]$ into $\D^\times[{\sf
t}^\times]$ is a complex vector space with a natural involution but it does not behave properly from the point of view of multiplication. Let, in fact, $A,B\in \LDD$ and assume that there exists a locally convex space $\E$ such that $\D \hookrightarrow \E \hookrightarrow \D^\times$. Assume, for instance, that $B: \D \to \E$ continuously and that $A$ has a continuous extension $\widetilde A$ to $\E$, then one may put
 $$ (A\cdot B) \xi := {\widetilde A}(B\xi), \quad \forall \xi \in \D.$$

 This could be the starting for introducing a partial multiplication in $\LDD$. But as shown by K{\"u}rsten  in \cite{kuersten} with several examples, this product depends on the choice of the intermediate space $\E$ and so it is not well defined (we refer to a locally convex space $\E$ such that $\D \hookrightarrow \E \hookrightarrow \D^\times$ as to an {\em interspace}) .
  A possible outcome to this problem consists in giving more restrictive conditions for the possible choice of the interspace $\E$, giving rise to the notion of {\em multiplication framework} \cite{tratschinke1, tratschinke1, at_pip}.

The first question we will try to answer in this paper corresponds, in a sense, to reversing this point of view: starting from a family $\{\H_\alpha; \alpha \in
\F\}$  of Hilbert spaces indexed by a directed set $\F$, does there exist a RHS $(\D, \H, \D^\times)$ so that every  $\H_\alpha$'s is an interspace between $\D$ and $\D^\times$? The observation that in many examples the extreme spaces $\D$ and $\D^\times$ of a RHS are, respectively, the locally convex projective and inductive limits of a family of Hilbert spaces, suggests constructing the RHS we are looking for RHS by emulating the procedures that lead to projective and inductive limits of Hilbert spaces.

Let us consider a {\em directed
system of Hilbert spaces}, i.e. a family $\{\H_\alpha; \alpha \in
\F\}$  of Hilbert spaces indexed by a set $\F$ directed upward by
$\leq$, such that for every $\alpha, \beta \in \F$, with $\beta\geq
\alpha$, there exists a linear map $T_{\beta \alpha}:\H_\alpha \to
\H_\beta$ with the properties
\begin{itemize}
                \item[(a)] $T_{\beta \alpha}$ is injective;
               \item[(b)] $T_{\alpha\alpha}=I_\alpha$, the identity of $\H_\alpha$;
                \item[(c)] $T_{\gamma\alpha}=T_{\gamma\beta}T_{\beta\alpha}$, $\alpha\leq \beta\leq \gamma$.
              \end{itemize}
If the maps $T_{\beta\alpha}$ are isometries, then the usual construction of the inductive limit of the family  $\{\H_\alpha, T_{\beta \alpha}, \alpha, \beta \in \F, \beta\geq \alpha\}$   produces a Hilbert space \cite{kadison}. But if we relax this assumption, we may get more general locally convex spaces.

In this paper we will consider the case where the maps $T_{\beta\alpha}$  are contractions. Then, a {\em directed contractive system of Hilbert spaces} generates automatically two locally convex spaces $\D$ and $\D^\times$ in conjugate duality to each other (this pair of spaces will be called the {\em joint topological} limit of the system). This construction (outlined in Section \ref{sect_2})
follows essentially the usual steps of the standard procedure for getting inductive and projective limits of Hilbert spaces; what makes the difference in the present approach is just the realization that these two procedures can be done at once, when a directed contractive system of Hilbert spaces is given. 
It is worth mentioning that, if the index set $\F$ has, in addition, an order reversing involution $\alpha \to \bar \alpha$ (with a {\em self-dual} element $o=\bar o$ corresponding to the so-called {\em central} Hilbert space $\H_o$), then the algebraic inductive limit of the family gives rise to a nested Hilbert space, provided that $\H_{\bar\alpha}$ is the conjugate dual of $\H_\alpha$, but smaller space $\D$ contained as a dense subspace in all the Hilbert spaces of the family may fail to exist.

In Section \ref{sect_3} we will give sufficient conditions for the system $\{\H_\alpha, T_{\beta \alpha}, \alpha, \beta \in \F, \beta\geq \alpha\}$ to generate a RHS  $(\D ,\H, \D^\times)$ and we will prove the main results of this paper: the RHS constructed in this way coincides with the canonical RHS defined by a family ${\mc O}$ of closable operators defined on $\D$.

Finally, in Section \ref{sect4} we will consider operators of from $\D$ into $\D^\times$ that can
be obtained as inductive limits of bounded operators on the Hilbert
spaces $\H_\alpha$'s and show that they can be cast into a natural
structure of partial *-algebra \cite{pop-book}, obtained in a rather natural way from the construction of the joint
topological limit itself without any direct reference to interspaces.

\section{Joint topological limits of Hilbert spaces}\label{sect_2}

Let $\{\H_\alpha; \alpha \in \F\}$ be a family of Hilbert spaces indexed by a set $\F$ directed upward by $\leq$ (we denote by $\ip{\cdot}{\cdot}_\alpha$ and $\|\cdot\|_\alpha$, respectively, the inner product and the norm of $\H_\alpha$).
Suppose that for every $\alpha, \beta \in \F$, with $\beta\geq \alpha$, there exists a linear map $U_{\beta \alpha}:\H_\alpha \to \H_\beta$ with the properties
\begin{itemize}
                \item[(i)] $U_{\beta \alpha}$ is injective;
                \item[(ii)] $\|U_{\beta \alpha}\xi_\alpha\|_\beta \leq \|\xi_\alpha\|_\alpha, \quad \forall \xi_\alpha \in \H_\alpha$;
                \item[(iii)] $U_{\alpha\alpha}=I_\alpha$, the identity of $\H_\alpha$;
                \item[(iv)] $U_{\gamma\alpha}=U_{\gamma\beta}U_{\beta\alpha}$, $\alpha\leq \beta\leq \gamma$.
              \end{itemize}

The family $\{\H_\alpha, U_{\beta \alpha}, \alpha, \beta \in \F, \beta\geq \alpha\}$ is called a {\em directed contractive system of Hilbert spaces}.

In this section we will show, by modifying the procedure of  \cite[Ch.IV]{scaefer},
that every directed contractive system of Hilbert spaces $\{\H_\alpha, U_{\beta \alpha}, \alpha, \beta \in \F, \beta\geq \alpha\}$ produces two spaces $\D$ and $\D^\times$ in conjugate duality: the space $\D^\times$ is obtained as the inductive limit of the system, while $\D$ turns out to be the projective limit of the spaces $\H_\alpha$'s, with respect to the adjoint maps of the $U_{\beta \alpha}$'s.

\medskip Let
$\L$ be the set of all functions $(\xi_\alpha)$ on $\F$ such that $\xi_\alpha \in \H_\alpha$.

Let
$$\S:=\{(\xi_\alpha)\in \L; \exists \gamma \in \F:\, \xi_\beta = U_{\beta\alpha}\xi_\alpha;\,\, \beta \geq \alpha \geq \gamma\}.$$
Then $\S$ is a complex vector space.

Let $$\S_0:= \{ (\xi_\alpha)\in \S;\, \exists\, \overline{\alpha}\in \F: \xi_\beta=0 ,\; \forall \beta \geq \overline{\alpha}\}.$$

Given $\alpha \in \F$, we define a linear map $\Theta_\alpha:\H_\alpha \to \S$ as follows: when $\xi \in \H_\alpha$ we put $\Theta_\alpha\xi= (\xi_\beta)_{\beta\in \F}$ where
\begin{equation}\label{defn_theta}\xi_\beta = \left\{ \begin{array}{lc} U_{\beta\alpha} \xi & \mbox{if }\beta\geq \alpha \\
0& \mbox{otherwise.}\end{array} \right. \end{equation}

We notice that
\begin{itemize}
\item[(i$_1$)] $\Theta_\alpha\xi \not\in \S_0$ if $\xi$ is a non-zero element of $\H_\alpha$;
\item[(i$_2$)] $\Theta_\alpha\xi- \Theta_\beta U_{\beta\alpha}\xi \in \S_0$ if $\beta\geq \alpha$ and $\xi \in \H_\alpha$.

The first follows from $\Theta_\alpha\xi\in \S_0 \Rightarrow \exists \overline{\alpha}\mbox{ such that }  U_{\beta\alpha}\xi=0, \, \forall \beta \geq \alpha, \overline{\alpha}$ and, since $U_{\beta\alpha}$ is injective, we get $\xi=0$.

As for (i$_2$), observe that if $\gamma\geq \beta$, for the $\gamma$
components we have $(\Theta_\alpha \xi)_\gamma =U_{\gamma\alpha}\xi$
and  $(\Theta_\beta
U_{\beta\alpha}\xi)_\gamma=U_{\gamma\beta}U_{\beta\alpha}\xi =
U_{\gamma\alpha}\xi$. Hence, $\Theta_\alpha\xi- \Theta_\beta
U_{\beta\alpha}\xi \in \S_0$.

\end{itemize}
Moreover, $\S_0$ is a subspace of $\S$. Put ${\mc E}:=\S/\S_0$. If $(\xi_\alpha)\in \S$, we denote by $\jj{(\xi_\alpha)}$ the corresponding coset; i.e., $\jj{(\xi_\alpha)}= (\xi_\alpha)+\S_0$.

Define $\btheta_\alpha: \xi \in \H_\alpha \to \jj{{\Theta_\alpha\xi}}\in {\mc E}$.
 The linear map $\btheta_\alpha$ is injective by (i$_1$). If we put $\|\btheta_\alpha \xi\|_\alpha := \|\xi\|_\alpha$, $\xi \in \H_\alpha$, then $\btheta_\alpha (\H_\alpha)$ is a Hilbert space contained in ${\mc E}$ as a subspace, $\btheta_\alpha$ is an isometric isomorphism and $\btheta_\alpha (\H_\alpha) \subseteq \btheta_\beta (\H_\beta)$ if $\beta \geq \alpha$. Since $\F$ is directed,
the set $$ \D^\times := \bigcup_{\alpha \in \F} \btheta_\alpha (\H_\alpha) $$ is a vector subspace of ${\mc E}$.
On $\D^\times$ it is defined in natural way the inductive topology $\sf{t_i}$, i.e.,
the finest locally convex topology such that every $\btheta_\alpha$ is continuous from $\H_\alpha$ into $\D^\times$.
 Thus, $\D^\times =\varinjlim \btheta_\alpha \H_\alpha$, the inductive limit of the family $\{\H_\alpha, U_{\beta \alpha}, \alpha, \beta \in \F, \beta\geq \alpha\}$.


\bigskip

Let $\{\H_\alpha, U_{\beta \alpha}, \alpha, \beta \in \F, \beta\geq \alpha\}$ be a directed contractive system of Hilbert spaces. Every linear map $U_{\beta \alpha}: \H_\alpha \to \H_\beta$, $\beta\geq \alpha$ is bounded, then it has an adjoint $U_{\beta \alpha}^*: \H_\beta \to \H_\alpha$, satisfying
$$\ip{U_{\beta \alpha}\xi_\alpha}{\eta_\beta}_\beta= \ip{\xi_\alpha}{U_{\beta\alpha}^*\eta_\beta}_\alpha, \; \forall \xi_\alpha \in \H_\alpha,\,\eta_\beta \in \H_\beta, \; \beta \geq \alpha.$$
We put $V_{\alpha\beta}:=U_{\beta\alpha}^*$, $\alpha\leq \beta$.
Then we have
\begin{itemize}
                \item[(i)] $V_{\alpha\beta}$ is injective;
                \item[(ii)] $\|V_{\alpha\beta}\xi_\beta\|_\alpha \leq \|\xi_\beta\|_\beta, \quad \forall \xi_\beta \in \H_\beta$;
                \item[(iii)] $V_{\alpha\alpha}=I_\alpha$, the identity of $\H_\alpha$;
                \item[(iv)] $V_{\alpha\gamma}=V_{\alpha\beta}V_{\beta\gamma}$, $\alpha\leq \beta\leq \gamma$.
              \end{itemize}

\medskip

Now, let
$$\D :=\{ (\xi_\beta) \in \L; \xi_\alpha=V_{\alpha\beta}\xi_\beta, \; \forall \alpha, \beta\in \F, \alpha \leq \beta\}$$
Define $\Pi_\alpha: \D \to \H_\alpha$ as the projection of $\D$  onto $\H_\alpha$; i.e., $\Pi_\alpha (\xi_\beta)=\xi_\alpha$, whenever $(\xi_\beta) \in \D$.
The map $\Pi_\alpha$
 is injective. Indeed, if $\Pi_\alpha(\xi_\beta)=\xi_\alpha=0$, then $\xi_\delta=0$ for $\delta\leq\alpha$, since $\xi_\delta=V_{\delta\alpha}\xi_\alpha$. If $\delta \not\leq \alpha$, we take $\gamma\geq \delta, \alpha$ then $0=\xi_\alpha=V_{\alpha\gamma}\xi_\gamma$, which implies that $\xi_\gamma=0$ by the injectivity of $V_{\alpha\gamma}$. Thus,  $\xi_\delta=V_{\delta\gamma}\xi_\gamma=0$.  Hence $\D$ can be identified with a subspace of $\H_\alpha$, for every $\alpha \in \F$. It is clear from the definition that $\Pi_\alpha= V_{\alpha\beta}\circ \Pi_\beta$. The space $\D$ can be equipped with the projective topology $\sf{t_p}$ of the $\H_\alpha$'s and so $\D=\varprojlim \H_\alpha$, the projective limit of the family $\{\H_\alpha, U_{\beta \alpha}, \alpha, \beta \in \F, \beta\geq \alpha\}$.

\bigskip
The fact that $\D$ and $ \D^\times$ can be put in conjugate duality relies on the following two observations.
\begin{enumerate}
\item[(D1)] If $(\xi_\alpha)\in \D$ and $(\eta_\alpha) \in \S_0$ then $\displaystyle \lim_\alpha \ip{\xi_\alpha}{\eta_\alpha}_\alpha=0$,
since $\eta_\beta=0$, for $\beta$ large enough.
\item[(D2)] If $(\xi_\alpha)\in \D$ and $(\eta_\alpha)\in \Theta_\gamma(\H_\gamma)$, then $\xi_\alpha =V_{\alpha\beta}\xi_\beta$ for every $\beta \geq \alpha$ and, on the other hand, $\eta_\delta=U_{\delta\gamma}\eta_\gamma$ for $\delta\geq \gamma$. Then for $\delta\geq \gamma$, we have
    $$\ip{\xi_\gamma}{\eta_\gamma}_\gamma-\ip{\xi_\delta}{\eta_\delta}_\delta=\ip{V_{\gamma\delta}\xi_\delta}{\eta_\gamma}_\gamma-
    \ip{\xi_\delta}{U_{\delta\gamma}\eta_\gamma}_\delta =0,$$
    since $V_{\gamma\delta}=U_{\delta\gamma}^*$. This means that, if $(\xi_\alpha)\in \D$ and $(\eta_\alpha)\in \Theta_\gamma(\H_\gamma)$, the net $\{\ip{\xi_\alpha}{\eta_\alpha}_\alpha\}$ is constant for $\alpha \geq \gamma$. Hence $\displaystyle\lim_\alpha \ip{\xi_\alpha}{\eta_\alpha}_\alpha$ always exists.
\end{enumerate}
Now let $(\xi_\alpha)\in \D$ and $\jj{(\eta_\alpha)}\in \D^\times$.
Then we define
$$ B(\jj{(\eta_\alpha)},(\xi_\alpha) ):=\lim_\alpha \ip{\xi_\alpha}{\eta_\alpha}_\alpha.$$
The previous observations show that $B$ is a well-defined
sesquilinear map (linear in the second argument and linear conjugate
in the first).

\begin{lemma} \label{lemma 3.1} Let $\D$ and $\D^\times$ be constructed as above. The following statements hold.
\begin{itemize}
\item[(i)] The (conjugate) duality between $\D^\times$ and $\D$ is separating.
\item[(ii)] $\Pi_\alpha(\D)$ is dense in $\H_\alpha$, for every $\alpha\in \F$.
\item[(iii)] The conjugate dual of $\D[{\sf t_p}]$ is (isomorphic to) $\D^\times$.
\end{itemize}
\end{lemma}
{ \begin{proof}
(i): Assume that, for some $(\xi_\alpha) \in \D$ and for every $\jj{(\eta_\alpha)}\in \D^\times$,  $\lim_\alpha \ip{\xi_\alpha}{\eta_\alpha}_\alpha=0$. Then, by (D2) and  for sufficiently large $\beta$, $\ip{\xi_\beta}{\eta_\beta}_\beta=0$. Since every vector $\eta_\beta\in \H_\beta$ defines $\btheta_\beta(\eta_\beta)\in \D^\times$, $\eta_\beta$ is an arbitrary vector in $\H_\beta$. Hence $\xi_\beta=0$. But as seen above if $(\xi_\alpha) \in \D$ has $0$ as $\beta$-component then $(\xi_\alpha)=(0)$.\\
Assume now that, for some $\jj{(\eta_\alpha)}\in \D^\times$ and for
every $(\xi_\alpha) \in \D$ , $\lim_\alpha
\ip{\xi_\alpha}{\eta_\alpha}_\alpha=0$. This implies that, for
sufficiently large $\beta$, $\ip{\xi_\beta}{\eta_\beta}_\beta=0$.
Also in this case $\xi_\beta$ can be thought as a generic element of
$\H_\beta$. Hence $\eta_\beta=0$ for sufficiently large $\beta$.
Thus, $(\eta_\alpha)\in\S_0$.

(ii): This is an easy consequence of (i).

(iii): Fix $(\eta_\alpha)\in \S$ so that $\jj{(\eta_\alpha)}\in \D^\times$. Then
$$B(\jj{(\eta_\alpha)},(\xi_\alpha) )=\lim_\alpha \ip{\xi_\alpha}{\eta_\alpha}_\alpha $$
is defined, for every $(\xi_\alpha)\in \D$. Since there exists
$\delta \in \F$ such that
$\ip{\xi_\gamma}{\eta_\gamma}_\gamma=\ip{\xi_\delta}{\eta_\delta}_\delta$,
$\gamma\geq \delta$, we have
$$B(\jj{(\eta_\alpha)},(\xi_\alpha) )= \ip{\xi_\delta}{\eta_\delta}_\delta.$$
Hence
$$|B(\jj{(\eta_\alpha)},(\xi_\alpha) )|=|\ip{\xi_\delta}{\eta_\delta}_\delta|\leq \|\xi_\delta\|_\delta \|\eta_\delta\|_\delta.$$
This proves that every $\jj{(\eta_\alpha)}\in \D^\times$ defines a continuous linear functional on $\D[{\sf t_p}]$.
\\ Conversely, let $F$ be a continuous linear functional on $\D[{\sf t_p}]$. Then there exists $\beta \in \F$ and $C>0$ such that
$$| F((\xi_\alpha))|\leq C\|\xi_\beta\|_\beta , \quad \forall (\xi_\alpha)\in \D,$$
where $\xi_\beta=\Pi_\beta (\xi_\alpha)$.
Let us now define a linear functional $F_\beta$ on $\Pi_\beta(\D)$ by
$$ F_\beta (\xi_\beta)=F((\xi_\alpha)).$$
Then
$$|F_\beta (\xi_\beta)|\leq C\|\xi_\beta\|_\beta,\quad \forall \xi_\beta \in \Pi_\beta(\D).$$
Since $\Pi_\beta(\D)$ is dense in $\H_\beta$, $F_\beta$ extends to a bounded linear functional on $\H_\beta$. Thus, there exists $\eta_\beta\in \H_\beta$ such that
$$ F_\beta (\xi_\beta)= \ip{\xi_\beta}{\eta_\beta}_\beta.$$
Now let us consider $\btheta_\beta \eta_\beta\in \D^\times$. Then we
have, for $\beta$ large enough
$$ F((\xi_\alpha))= F_\beta(\xi_\beta) = \ip{\xi_\delta}{\eta_\delta}_\delta =B(\btheta_\beta \eta_\beta, (\xi_\beta) ).$$
This proves the statement. \end{proof}}

We summarize the previous discussion in the following
\begin{thm} \label{thm_main} Let $\{\H_\alpha, U_{\beta \alpha}, \alpha, \beta \in \F, \beta\geq \alpha\}$ be a directed contractive system of Hilbert spaces.

\begin{itemize}

\item[(d$_1$)] There exists a conjugate dual pair $(\D^\times, \D)$ and, for every $\alpha \in \F$, a pair of injective linear maps $(\Pi_\alpha, \btheta_\alpha)$,
$
 \Pi_\alpha: \D \to \H_\alpha,$
 $\btheta_\alpha: \H_\alpha \to \D^\times,
$
both with dense range, such that
\begin{itemize}
  \item[(I$_1$)]$\Pi_\alpha=V_{\alpha\beta}\circ \Pi_\beta, \; \alpha \leq \beta$ (where $V_{\alpha\beta}=U_{\beta\alpha}^*$);
  \item[(I$_2$)]$\btheta_\alpha= \btheta_\beta\circ U_{\beta\alpha}, \; \alpha \leq \beta$
  \item[(I$_3$)]$\D^\times = \bigcup_{\alpha \in \F} \btheta_\alpha (\H_\alpha).$
  \item[(I$_4$)]If $\xi \in \D$ and $\eta\in \D^\times$, with $\eta = \btheta_\alpha \eta_\alpha$, for some $\alpha \in \F$ and $\eta_\alpha \in \H_\alpha$,
  then
  $$ B(\eta, \xi)= B(\btheta_\alpha\eta_\alpha, \xi)=\ip{\Pi_\alpha\xi}{\eta_\alpha}_\alpha,$$
  independently of $\alpha$ such that $\eta \in \btheta_\alpha (\H_\alpha)$.
\end{itemize}
\item[(d$_2$)] The pair $(\D^\times, \D)$ occurring in (d$_1$) is uniquely determined by the conditions given in (d$_1$), in the following sense: if $(\D_1^\times, \D_1 )$ is another conjugate dual pair for which  there exists, for every $\alpha \in \F$, a pair $(\Delta_\alpha, {\mathbf \Gamma}_\alpha)$, $
 \Delta_\alpha: \D_1 \to \H_\alpha,$ ${\mathbf \Gamma}_\alpha: \H_\alpha \to \D_1^\times$ such that the statements corresponding to (I$_1$) -- (I$_4$) are satisfied, then there exists an injective linear map  $T: \D^\times \to \D_1^\times$ such that, for every $\alpha \in \F$, ${\mathbf \Gamma}_\alpha = T\circ \btheta_\alpha$ and $\Delta_\alpha = \Pi_\alpha\circ T^\times$, where $T^\times: \D_1 \to \D$ denotes the adjoint map of T (w.r. to the Mackey topology).
\end{itemize}
\end{thm}

\begin{proof} It remains to prove (d$_2$). It is easy to see that ${\mathbf \Gamma}_\alpha \circ {\btheta_\alpha}^{-1}$ is an isomorphism of $\btheta_\alpha (\H_\alpha)$ onto ${\mathbf \Gamma}_\alpha(\H_\alpha)$, for every $\alpha \in \F$.
Moreover, if $\xi_\alpha \in \H_\alpha$ and $\beta \geq \alpha$,
\begin{eqnarray*} ({\mathbf \Gamma}_\beta \circ {\btheta_\beta}^{-1})(\btheta_\alpha\xi_\alpha)&=&{\mathbf \Gamma}_\beta (U_{\beta \alpha}\xi_\alpha)\\
&=& {\mathbf \Gamma}_\alpha \xi_\alpha = ({\mathbf \Gamma}_\alpha \circ {\btheta_\alpha}^{-1})(\btheta_\alpha\xi_\alpha). \end{eqnarray*}
Hence ${\mathbf\Gamma}_\beta \circ {\btheta_\beta}^{-1}$ extends ${\mathbf \Gamma}_\alpha \circ {\btheta_\alpha}^{-1}$. Hence there exists an injective linear map $T$ from $\D=\bigcup_{\alpha \in \F}\btheta_\alpha(\H_\alpha)$ onto $\D_1=\bigcup_{\alpha \in \F}{\mathbf \Gamma}_\alpha(\H_\alpha)$ which coincides with ${\mathbf \Gamma}_\alpha \circ {\btheta_\alpha}^{-1}$ when restricted to $\btheta_\alpha(\H_\alpha)$. Clearly, ${\mathbf \Gamma}_\alpha = T\circ \btheta_\alpha$. Now, if we denote by $B$ and $B_1$ the sesquilinear forms defining the duality, respectively, of the pairs  $(\D^\times, \D)$ and $( \D_1^\times, \D_1,)$, we get, for $\xi_1 \in \D_1, \eta_1 \in \D_1^\times$ with $\eta_1 = {\mathbf\Gamma}_\alpha \eta_{1,\alpha}$, for some $\alpha \in \F$ and $\eta_{1,\alpha} \in \H_\alpha$,
\begin{eqnarray*}
B_1( \eta_1, \xi_1)&=& B_1( {\mathbf\Gamma}_\alpha \eta_{1,\alpha},\xi_1)= B_1( T\circ \btheta_\alpha \eta_{1,\alpha}, \xi_1) \\ &=& B(\btheta_\alpha \eta_{1,\alpha}, T^\times \xi_1)=\ip{(\Pi_\alpha\circ T^\times)\xi_1}{\eta_{1,\alpha}}_\alpha.
\end{eqnarray*}
On the other hand,
$$B_1( {\mathbf\Gamma}_\alpha \eta_{1,\alpha}, \xi_1)=\ip{\Delta_\alpha\xi_1}{\eta_{1,\alpha}}_\alpha.$$
These equalities imply that $\Delta_\alpha = \Pi_\alpha\circ T^\times$.
\end{proof}

The conjugate dual pair $(\D^\times,\D)$ occurring in Theorem \ref{thm_main} will be called the {\em joint topological limit} of the directed contractive system $\{\H_\alpha, U_{\beta \alpha}, \alpha, \beta \in \F, \beta\geq \alpha\}$ of Hilbert spaces.

\medskip The space $\D$ can be, of course, identified with a subspace of $\D^\times$. Indeed, for every $\alpha \in \F$, the map $\Lambda_\alpha:=\btheta_\alpha\circ \Pi_\alpha$ is linear and injective. So that $\Lambda_\alpha(\D)$ is a subspace of $\D^\times$, isomorphic to $\D$. If $(\xi_\alpha), (\eta_\alpha)\in \D$ we have
$$B( (\btheta_\beta\circ \Pi_\beta) (\eta_\alpha), (\xi_\alpha))=\lim_\gamma \ip{\xi_\gamma}{ U_{\gamma\beta}\eta_\beta}_\gamma= \lim_\gamma\ip{V_{\beta\gamma}\xi_\gamma}{\eta_\beta}_\beta =\ip{\xi_\beta}{\eta_\beta}_\beta$$
which makes clear the dependence on $\beta$ of the left hand side.
An unambiguous identification of $\D$ into $\D^\times$ is possible,
for instance, if there exists $\gamma\in \F$ such that
$\btheta_\alpha\circ \Pi_\alpha = \btheta_\beta\circ \Pi_\beta$, for
$\beta\geq \alpha\geq \gamma$.  This equality holds if, and only if,
the $V_{\alpha\beta}$'s ($\beta\geq \alpha$) are isometries
[Proposition \ref{prop_isometr}, below]. This assumption is,
however, too strong and not fulfilled in typical examples.

\begin{prop}\label{prop_isometr} The following statements are equivalent.
\begin{itemize}
\item[(i)] $\btheta_\alpha\circ \Pi_\alpha = \btheta_\beta\circ \Pi_\beta$, for  $\beta\geq \alpha$.
\item[(ii)]$V_{\alpha\beta}$ is an isometry, for  $\beta\geq \alpha$.
\end{itemize}
\end{prop}
\begin{proof}  If $\beta\geq \alpha$,  then, by Theorem  \ref{thm_main},
\begin{equation}\label{eq_1} \btheta_\alpha \circ \Pi_\alpha =  (\btheta_\beta \circ U_{\beta\alpha})\circ (V_{\alpha\beta}\circ\Pi_\beta)= \btheta_\beta \circ ( U_{\beta\alpha}\circ V_{\alpha\beta})\circ \Pi_\beta.\end{equation}
So if $U_{\beta\alpha}V_{\alpha\beta}=I_\beta$, we get $\btheta_\alpha\circ \Pi_\alpha = \btheta_\beta\circ \Pi_\beta$.
On the other hand, if $\btheta_\alpha\circ \Pi_\alpha = \btheta_\beta\circ \Pi_\beta$, for $\beta\geq \alpha$, then, by \eqref{eq_1}, we have
$$ \btheta_\beta \circ ( U_{\beta\alpha}\circ V_{\alpha\beta})\circ \Pi_\beta = \btheta_\beta\circ \Pi_\beta.$$
Since both $\btheta_\beta$ and $\Pi_\beta$ are injective, we easily obtain $U_{\beta\alpha}V_{\alpha\beta}=I_\beta$; i.e., $V_{\alpha\beta}$ is an isometry.
\end{proof}

\section{Rigged Hilbert spaces as joint topological limit}\label{sect_3}
In this section we will discuss the possibility that the joint topological limit $(\D, \D^\times)$ of a directed contractive family $\{\H_\alpha, U_{\beta \alpha}, \alpha, \beta \in \F, \beta\geq \alpha\}$ of Hilbert spaces gives rise to a rigged Hilbert space. Some explanation is here in order. As we have seen (we maintain the notations of Section \ref{sect_2}), every space $\H_\alpha$ can be identified, by means of the map $\btheta_\alpha$, with a subspace of $\D^\times$ as well as $\D$ can be identified with a subspace of $\H_\alpha$ by $\Pi_\alpha$. Then, clearly, the triplet $(\Pi_\alpha(\D), \H_\alpha, \D^\times)$ can be regarded as a RHS.
To be more definite we give the following
\bedefi Let  $(\D, \D^\times)$ be the joint topological limit of a directed contractive family $\{\H_\alpha, U_{\beta \alpha}, \alpha, \beta \in \F, \beta\geq \alpha\}$ of Hilbert spaces. We say that $\D, \D^\times$ are the extreme spaces of a RHS if there exists a Hilbert space $\H_0$, with inner product $\ip{\cdot}{\cdot}_0$, with the properties:
\begin{itemize}
\item[(i)] $\D$ is a dense subspace of $\H_0$;
\item[(ii)] for every $\alpha \in \F$, there exists an injective linear map $\sigma_\alpha:\H_\alpha\to \H_0$ such that $\sigma_\beta= \sigma_\alpha V_{\alpha\beta} $ if $\alpha \leq \beta $ and $\bigcup_{\alpha \in \F}\sigma_\alpha(\H_\alpha)$ is dense in $\H_0$;
\item[(iii)] for every $\eta \in \H_0$, the conjugate linear functional $F_\eta: \xi \to \ip{\eta}{\xi}_0$ is continuous on $\D[\sf{t_p}]$ and the linear map $J: \eta \in \H_0 \to F_\eta \in \D^\times$ is injective.
\end{itemize}
\findefi

In order to give a sufficient condition for $(\D, \D^\times)$ to be the extreme spaces of a RHS, we assume that
\begin{itemize}
\item[(A)] $\quad \{ (\xi_\alpha)\in \D:\, \inf_{\alpha\in \F} \|\xi_\alpha\|_\alpha=0\}=\{(0)\}.$ \end{itemize}
Then we put, for $(\xi_\alpha)\in \D$,
$$ \|(\xi_\alpha)\|_0=\inf_{\alpha\in \F}\|\xi_\alpha\|_\alpha.$$
Then it is easily seen that $\|\cdot\|_0$ is a norm on $\D$ and that it satisfies the parallelogram law. Hence it is possible to define on $\D$ an inner product $\ip{\cdot}{\cdot}_0$, which makes $\D$ into a pre-Hilbert space.
We denote by $\H_0$ the Hilbert space completion of $\D$.

\berem Assume $\{(\xi_\gamma)_n\}$ is a sequence in $\D$ such that $\|\cdot\|_0-\lim_{n\to \infty}(\xi_\gamma)_n=0$. We denote by $\{\xi_\gamma^n\}$ the corresponding sequence of $\gamma$ components of $\{(\xi_\gamma)_n\}$. Then, as it is easily seen, there exists $\overline{\alpha}\in \F$ such that $\xi_{\delta}^n \to 0$ in $\H_{\delta}$, for every $\delta\leq \overline{\alpha}$. But we have no information about the limit in $\H_\beta$, when $\beta \not \leq \overline{\alpha}$. For this reason we introduce the following condition (C), which expresses the {\em compatibility} of the norms.
\enrem
\begin{itemize}
\item[(C)] If $\{(\xi_\gamma)_n\}$ is a sequence in $\D$ such that $\|\cdot\|_0-\lim_{n\to \infty}(\xi_\gamma)_n=0$ and the sequence of $\alpha$ components $\{\xi_\alpha^n\}$ is Cauchy, for some $\alpha\in \F$,  then $\displaystyle\|\cdot\|_\alpha-\lim_{n\to\infty}\xi_\alpha^n =0$.
\end{itemize}
\begin{lemma}\label{lem_one} If condition (C) holds, then $\H_\alpha$ is isomorphic to a subspace of $\H_0$.
\end{lemma}
\begin{proof} Let $\xi_\alpha \in \H_\alpha$. By (ii) of Lemma \ref{lemma 3.1}, there exists a sequence $\{(\xi_\gamma)_n\}$ in $\D$ such that $\Pi_\alpha (\xi_\gamma)_n \to \xi_\alpha$; i.e., $\xi_\alpha^n \to \xi_\alpha$.
On the other hand, if $\epsilon>0$,
$$ \inf_\beta\|\xi_\beta^n - \xi_\beta^m\|_\beta \leq \|\xi_\alpha^n - \xi_\alpha^m \|_\alpha <\epsilon , \quad\mbox{for } n, m \mbox{ large enough}.$$
Thus $\{(\xi_\gamma)_n\}$ is Cauchy w.r. to $\|\cdot \|_0$. Let $\zeta_\alpha$ denote its limit in $\H_0$. It is easily seen that $\zeta_\alpha$ does not depend on the sequence $\{(\xi_\gamma)_n\}$ where we started from. Hence, we define the linear map $\sigma_\alpha$ by
$$ \sigma_\alpha: \xi_\alpha \in \H_\alpha \to \zeta_\alpha \in \H_0.$$
Now we show that $\sigma_\alpha$ is injective. Assume that
$\zeta_\alpha=0$. By definition, $\zeta_\alpha$ is the
$\|\cdot\|_0$-limit of a sequence $\{(\xi_\gamma)_n\}$ in $\D$ such
that $\Pi_\alpha (\xi_\gamma)_n =\xi_\alpha^n \to \xi_\alpha$. The
sequence $\{\xi_\alpha^n\}$ is then Cauchy and so, by (C),
$\xi_\alpha=0$. If $\gamma\leq \alpha$, then using the equality
$\Pi_\gamma=V_{\gamma\alpha}\circ\Pi_\alpha$ and the continuity of
$V_{\gamma\alpha}$ one easily proves that $\zeta_\gamma=
V_{\gamma\alpha} \zeta_\alpha$. Hence $\sigma_\alpha= \sigma_\gamma
V_{\gamma\alpha}$, if $\gamma\leq \alpha$.
\end{proof}

\begin{lemma}\label{lem_two} $\H_0$ can be identified with a subspace of $\D^\times$.
\end{lemma}
\begin{proof} Let $\eta \in \H_0$. Then if $\xi=(\xi_\alpha)\in \D$ we have
$$ |\ip{\xi}{\eta}_0|\leq \|\xi\|_0\|\eta\|_0 \leq \|\xi_\alpha\|_\alpha \|\eta\|_0, \quad \forall \alpha \in \F.$$
Hence $\eta$ defines a conjugate linear functional $F_\eta$, continuous for the projective topology ${\sf t_p}$ of $\D$. The map $J:\eta \in \H_0 \to \F_\eta \in \D^\times$ is injective and, hence, $\H_0$ is identified with a subspace of $\D^\times$.
\end{proof}
Then, finally, we obtain
\begin{thm}\label{thm_rhs} If the joint topological limit $(\D, \D^\times)$  of a directed contractive family $\{\H_\alpha, U_{\beta \alpha}, \alpha, \beta \in \F, \beta\geq \alpha\}$ of Hilbert spaces satisfies the conditions (A) and (C), then $\D, \D^\times$ are the extreme spaces of a RHS.
\end{thm}

\beex Let us consider the family of Hilbert spaces $\{\H_\alpha\}$ where $\H_\alpha:= L^2([\alpha, +\infty), dx)$, $\alpha \in {\mb R}$.
We define
$U_{\beta\alpha}: L^2([\alpha, +\infty)) \to L^2([\beta, +\infty))$ by
$$(U_{\beta\alpha}f)(x)=f(x-\beta+\alpha),\; \;f \in L^2([\alpha, +\infty)),\, \beta\geq \alpha.$$
Then, as it is easily checked,  the $U_{\beta\alpha}$'s are unitary
operators. The procedure outlined produces in this case the usual
inductive limit and then it gives as final result only one Hilbert
space, $\D= \D^\times =L^2 (\mb R)$. \enex \beex A more interesting
example is obtained by considering the family of Hilbert spaces
$\{\H_\alpha\}$ with
$\H_\alpha:= L^2({\mb R}, (1+|x|^\alpha) dx)$, $\alpha \in {\mb R}^+\cup \{0\}$.\\
We define
$U_{\beta\alpha}: L^2({\mb R}, (1+|x|^\alpha) dx) \to L^2({\mb R}, (1+|x|^\beta) dx))$ by
$$(U_{\beta\alpha}f)(x)= \frac{1+|x|^\alpha}{1+|x|^\beta}f(x),\;\; f \in L^2({\mb R}, (1+|x|^\alpha) dx),\, \beta\geq \alpha.$$
In this case the $U_{\beta\alpha}$'s are only contractions and $V_{\alpha\beta}=I_{\alpha\beta}$ (the identity of $\H_\beta$ into $\H_\alpha$), as can be easily computed.
The space $\D$ can be identified with the space $\D^\infty (Q)$ where $Q$ is the operator of multiplication by $x$ in $L^2({\mb R}, dx)$ defined in the dense domain $D(Q)=\{f \in L^2({\mb R}, dx): xf \in L^2({\mb R}, dx)\}$. Indeed,
\begin{eqnarray*} \D &=& \bigcap_{\alpha \in {\mb R}^+\cup \{0\}}L^2({\mb R}, (1+|x|^\alpha) dx) \\
&=& \bigcap_{\alpha \in {\mb R}^+\cup \{0\}}D((I+|Q|^\alpha)^{1/2})= \bigcap_{n \in {\mb N}}D(Q^n)= \D^\infty (Q).\end{eqnarray*}
The space $\D^\times$ can be described as follows:
$$\D^\times = \{f \mbox{ measurable }:\, (I+|Q|^\alpha)^{-1/2} f \in L^2({\mb R}, dx)\}.$$
Both conditions (A) and (C) are satisfied in this example, so, according to Theorem \ref{thm_rhs}, the construction of the joint topological limit generates the rigged Hilbert space
$\D \hookrightarrow L^2({\mb R}, dx)\hookrightarrow \D^\times$.
\enex

\beex \label{ex_main}Let $\D_0$ be a dense domain in Hilbert space $\H$, with norm $\| \cdot\|$, and ${\mc O}$ a family of closable operators
 with domain $\D_0$; in this case we say that ${\mc O}$ is an {\em O-family}.
 Assume that ${\mc O}$ is a directed set for the order relation $A \preceq B \, \Leftrightarrow\, \|A\xi\|\leq \|B\xi\|, \forall \xi \in \D_0$. Let us consider the family of Hilbert spaces $\{\H_A;\, A \in {\mc O}\}$, where $\H_A$ is the completion of $\D_0$ under the norm
$$\|\xi\|_A =\sqrt{\|\xi\|^2 + \|A\xi\|^2}= \|(I+A^*\overline{A})^{1/2}\xi\|,\quad \xi \in \D_0.$$ It is easily seen that if $A \preceq B$, $\H_B\subseteq \H_A$.\\
  For shortness, we put  $S_A:=I+A^*\overline{A}$. Then, $S_A$ has a bounded inverse $S_A^{-1}$. Moreover, for every $A,B \in {\mc O}$ with $B\succeq A$,  the operator $S_A^{1/2}S_B^{-1/2}$ is also bounded.\\
 Now we need to define the maps $U_{BA}$, when $A \preceq B$.

For this, let us consider an element $\xi_A \in \H_A$. Then $\xi_A$  defines a bounded conjugate linear
functional $F_{\xi_A}$ on $\H_A$ by
$$F_{\xi_A}(\eta)=\ip{\xi_A}{\eta}_A,\quad\eta\in \H_A.$$
If now $\eta\in \H_B$, $B\succeq A$, we have
$$|F_{\xi_A}(\eta)|=|\ip{\xi_A}{\eta}_A|\leq\|\xi_A\|_A\|\eta\|_A\leq\|\xi_A\|_A\|\eta\|_B.$$
Hence $F_{\xi_A}\upharpoonright\H_B$ is a bounded conjugate linear
function on $\H_B$, thus there exists a unique $\xi_B\in\H_B$ such
that
$$\ip{\xi_A}{\eta}_A=\ip{\xi_B}{\eta}_B\qquad\forall\eta\in\H_B;$$
it result also that $\|\xi_B\|_B\leq\|\xi_A\|_A$.

We define $U_{BA}\xi_A:=\xi_B.$ Then $U_{BA}$ ($A \preceq B$) is a contraction; indeed,
$\|U_{BA}\xi_A\|_B=\|\xi_B\|_B\leq\|\xi_A\|_A.$

Now we show that, for $A \preceq B$,
$$ U_{BA}= S_B^{-1/2}(S_A^{1/2}S_B^{-1/2})^*S_A^{1/2}.$$

Indeed, we have
\begin{eqnarray*}
\ip{\xi_A}{\eta}_A &=&  \ip{S_A^{1/2}\xi_A}{S_A^{1/2}\eta}=\ip{S_A^{1/2}\xi_A}{S_A^{1/2}(S_B^{-1/2}S_B^{1/2})\eta}=\\
 &=&\ip{S_A^{1/2}\xi_A}{(S_A^{1/2}S_B^{-1/2})(S_B^{1/2}\eta)}\\
  &=&\ip{(S_B^{1/2}S_B^{-1/2})(S_A^{1/2}S_B^{-1/2})^*S_A^{1/2}\xi_A}{S_B^{1/2}\eta}\\
   &=&\ip{S_B^{-1/2}(S_A^{1/2}S_B^{-1/2})^*S_A^{1/2}\xi_A}{\eta}_B.
\end{eqnarray*}

 Then, $U_{BA}:\H_A \to \H_B$, for $B\succeq A$, and satisfies the conditions given at beginning of Section 2; hence, the family $\{\H_A, U_{BA}; A,B \in  {\mc O}, A \preceq B\}$ is a directed contractive family of Hilbert spaces.
It is easily seen that $V_{AB}=I_{AB}$, the identity operator
from $\H_B$ into $\H_A$. \\

Let $\D=\bigcap_{A\in {\mc O}}\H_A$ be endowed with the
graph topology ${\sf t}_{\mc O}$ defined by the family of
(semi)norms $\{\|\cdot\|_A, \, A \in {\mc O}\}$  and let
$\D^\times$ be the conjugate dual of $\D[{\sf t}_{\mc O}]$. Then, as it is well known, ${\sf t}_{\mc O}$ is nothing but the the projective topology defined by the spaces
$\H_A[\|\cdot\|_A]$ and $\D^\times$ can be viewed as the inductive limit of their conjugate duals $\H_A^\times[\|\cdot\|_A^\times]$, where $\|\cdot\|_A^\times$ stands for the dual norm.

An explicit form of the map $\btheta_A$ which embeds $\H_A$ into $\D^\times$, can be easily obtained by using the fact that every operator $S_A^{1/2}$, $A\in {\mc O}$, which is continuous from $\H_A[\|\cdot\|_A]$ into $\H[\|\cdot\|]$ has a continuous adjoint $(S_A^{1/2})^\times: \H[\|\cdot\|] \to \H_A^\times[\|\cdot\|_A^\times] $ (extending the Hilbert adjoint $(S_A^{1/2})^*= S_A^{1/2}$).
Indeed, let $\xi_A\in\H_A$ and consider the functional
$$G_{\xi_A}(\eta)=\ip{\xi_A}{\eta}_A,\qquad\eta\in\D.$$ Since $G_{\xi_A}$ is
continuous in the projective topology  $\sf{t_p}$, there exists
a unique $\bar{\xi}_A\in\D^\times$ such that
$$\ip{\bar{\xi}_A}{\eta}=\ip{\xi_A}{\eta}_A=\ip{S_A^{1/2}\xi_A}{S_A^{1/2}\eta}\qquad \forall\eta\in\D.$$
Then $\btheta_A \xi_A:= \bar{\xi}_A$ defines the natural embedding of $\H_A$ into $\D^\times$.
Taking into account the previous remark on the operator $S_A^{1/2}$,
we can then write $\btheta_{A}\xi_A=(S_A^{1/2})^\times
S_A^{1/2}\xi_A = \bar{\xi}_A \in\H_A^\times \subset \D^\times$. Further, from
\begin{eqnarray*}
  \ip{(S_A^{1/2})^\times S_A^{1/2}\xi_A}{\eta}&=& \ip{S_A^{1/2}\xi_A}{S_A^{1/2}\eta}= \ip{U_{BA}\xi_A}{\eta}_B\\
&=&\ip{S_B^{1/2}U_{BA}\xi_A}{S_B^{1/2}\eta}=\ip{(S_B^{1/2})^\times
S_B^{1/2}U_{BA}\xi_A}{\eta}.
\end{eqnarray*}
it follows that
$\btheta_{A}=(S_A^{1/2})^\times S_A^{1/2}=(S_B^{1/2})^\times S_B^{1/2}U_{BA}=\btheta_{B}U_{BA}$.
 \enex
The situation described in Example \ref{ex_main} is indeed the most
general when a directed contractive family of Hilbert spaces defines
a rigged Hilbert space, as the next theorem shows.
\begin{thm}
Let $(\D, \D^\times)$ be the joint topological limit  of a directed contractive family $\{\H_\alpha, U_{\beta \alpha}, \alpha, \beta \in \F, \beta\geq \alpha\}$ of Hilbert spaces. Assume that the conditions (A) and (C) are satisfied. Denote by $\H_0[\|\cdot\|_0]$ the Hilbert space which makes of $(\D, \H_0, \D^\times)$ a rigged Hilbert space. Then the following statements hold.
\begin{itemize}
  \item[(i)] For every $\alpha \in \F$, there exists a linear operator $A_\alpha$ with domain $\D$, closable in $\H_0$ such that $\H_\alpha$ is the completion $\H_{A_\alpha}$ of $\D$ under the norm $\|\xi\|_{A_\alpha}=\|(I+A_\alpha^*\overline{A}_\alpha)^{1/2}\xi\|_0$, $\xi \in \D$.
  \item[(ii)] The family ${\mc O}=\{A_\alpha,\, \alpha \in \F\} $ is directed upward by $\F$ (i.e., $\alpha \leq \beta \Leftrightarrow A_\alpha\preceq A_\beta$).
  \item[(iii)]$\D=\bigcap_{\alpha \in \F}\H_{A_\alpha} $ and $\D^\times=\bigcup_{\alpha \in \F}\btheta_{A_\alpha}(\H_{A_\alpha})$ is the conjugate dual of $\D$ for the graph topology ${\sf t}_{\mc O}$. The inductive topology of $\D^\times$ coincides with the Mackey topology $\tau(\D^\times, \D)$.
\end{itemize}
\end{thm}
\begin{proof} (i): Since $\|\cdot\|_0=\inf_\alpha \|\cdot\|_\alpha$, it follows that, for every $\alpha \in \F$, the inner product $\ip{\cdot}{\cdot}_\alpha$ of $\H_\alpha$ can be viewed as a closed positive sesquilinear form defined on $\H_\alpha \times \H_\alpha \subset \H_0\times \H_0$ (up to an isomorphism) which, as a quadratic form on $\H_\alpha$, has $1$ as greatest lower bound. Then there exists a selfadjoint operator $B_\alpha$ with dense domain $D(B_\alpha)$ in $\H_0$, with $B_\alpha \geq I$, such that
\begin{align*}
& D(B_\alpha)=\H_\alpha \\
& \ip{\xi}{\eta}_\alpha = \ip{B_\alpha\xi}{B_\alpha\eta}_0, \quad
\forall \xi, \eta \in \H_\alpha.
\end{align*}
Since (the image of) $\D$ is dense in $\H_\alpha$, $\D$ is a core for $B_\alpha$ and, hence, also for the operator $(B_\alpha^2-I)^{1/2}$. We define $A_\alpha = (B_\alpha^2-I)^{1/2}\upharpoonright \D$.
The proofs of (ii) and (iii) follow, now, from simple considerations. In particular, the fact that the inductive topology of $\D^\times$ coincides with the Mackey topology $\tau(\D^\times, \D)$ is a well-known fact \cite[Ch. IV, Sec. 4.4]{scaefer} or \cite[Sec. 2.3]{at_pip}.
\end{proof}

\section{Inductive limit of operators}\label{sect4}
Once the joint topological limit of a family of Hilbert spaces is at hand, it is natural to look at operators acting on it and characterize those which can be expressed as inductive limits of bounded operators on the Hilbert spaces entering the construction.

\medskip Let $(\D, \D^\times)$ be the joint topological limit  of a directed contractive family \linebreak $\{\H_\alpha, U_{\beta \alpha}, \alpha, \beta \in \F, \beta\geq \alpha\}$ of Hilbert spaces.
We denote by $\LBDD{}$ the space of all linear maps $X:\D\to\D^\times$ such that there exists $\gamma\in \F$ and $C>0$ such that
\begin{equation}\label{eq_ineduality}|B(X(\eta_\alpha), (\xi_\alpha))|\leq C \|\xi_\gamma\|_\gamma \|\eta_\gamma\|_\gamma , \quad \forall (\xi_\alpha), (\eta_\alpha)\in \D.\end{equation}

\begin{prop}\label{prop_operators1}
Assume that, for each $\alpha \in \F$, an operator $X_\alpha\in \BB(\H_\alpha)$ is given and that there exist $\overline \alpha\in \F$ for which  $X_\beta= U_{\beta\alpha}X_\alpha V_{\alpha\beta}$, $\overline{\alpha}\leq \alpha \leq \beta$. Then there exists a unique linear map $X \in \LBDD{}$ such that $X(\xi_\gamma)= \btheta_\beta X_\beta\Pi_\beta(\xi_\gamma)$ whenever $\beta\geq \overline{\alpha}$.
The map $X$ is called the inductive limit of the operators $X_\alpha$ and denoted by $X=\varinjlim X_\alpha$.
\end{prop}
\begin{proof} Let $(\xi_\gamma) \in \D$. Consider the family of vectors $(\xi'_\gamma)$ with

$$ \xi'_\gamma:=\left\{  \begin{array}{ll} X_\gamma\xi_\gamma & \mbox{if } \gamma\geq \overline{\alpha} \\
  0 & \mbox{otherwise.}
 \end{array}
\right. $$
Then,
$(\xi'_\gamma)\in \S$. Indeed, if $\gamma\geq \beta\geq \overline{\alpha}$,  $\xi'_\gamma = X_\gamma \xi_\gamma=U_{\gamma\beta}X_\beta V_{\beta\gamma}\xi_\gamma=U_{\gamma\beta}X_\beta\xi_\beta.$

If $\xi_\alpha\in \H_\alpha$, then $X_\alpha \xi_\alpha \in \H_\alpha$ and
$\Theta_\alpha X_\alpha\xi_\alpha= (\eta_\gamma)$, with $\eta_\gamma$ defined as in \eqref{defn_theta}.\\
So that, if $\beta\geq \alpha \geq \overline{\alpha}$,
$$\eta_\beta =\left\{
                                    \begin{array}{ll}
                                      U_{\beta\alpha}X_\alpha V_{\alpha\beta}\xi_\beta & \mbox{if }\beta\geq \alpha \\
                                      0 & \mbox{otherwise.}
                                    \end{array}
                                  \right.$$
Hence, $\xi'_\beta= \eta_\beta$, $\beta\geq \overline{\alpha}$. Therefore, if we define
$ X(\xi_\beta)= \jj{(\eta_\gamma)}$, $X$ is a linear map from $\D$ into $\D^\times$ and
$$ X(\xi_\beta)= \btheta_{\alpha}X_\alpha\Pi_\alpha (\xi_\beta),\quad \beta,\,\alpha \geq \overline{\alpha}.$$
 Moreover,
$$|B(X(\xi_\beta),(\zeta_\beta))|=\lim_\alpha |\ip{\zeta_\alpha}{\eta_\alpha}_\alpha|\leq \|X_\delta\|_{\delta\delta}\|\zeta_\delta\|_\delta\|\xi_\delta\|_\delta, \; (\zeta_\beta),(\xi_\beta)\in \D,\, \delta\geq \overline{\alpha},$$
where $\|X_\delta\|_{\delta\delta}$ denotes the norm of $X_\delta$ in $\BB(\H_\delta)$, $\delta \in \F.$\\
This implies that $X \in \LBDD{}$. The uniqueness is clear, so the  statement is proved.
\end{proof}

\begin{prop}\label{prop_operators2} Let $X \in \LBDD{}$ and ${\sf d}(X)$ the set of $\gamma \in \F$ for which the inequality \eqref{eq_ineduality} holds.  Then for each $\gamma \in {\sf d}(X)$ there exists a bounded operator $X_\gamma$ in $\H_\gamma$ such that
$$ B(X(\eta_\alpha), (\xi_\alpha)) = \ip{\xi_\gamma}{X_\gamma\eta_\gamma}_\gamma, \; \xi_\gamma, \eta_\gamma \in \H_\gamma.$$
Putting $X_\beta =0$ if $\beta \not\in {\sf d}(X)$, then $X=\varinjlim X_\gamma$.
\end{prop}
\begin{proof} The existence of $X_\gamma$ follows immediately from the representation theorem of bounded sesquilinear forms, once one has defined
$$ F_X^\gamma(\xi_\gamma, \eta_\gamma):= B(X(\eta_\alpha), (\xi_\alpha)), \quad \xi_\gamma, \eta_\gamma \in \H_\gamma.$$
The sesquilinear form $F_X^\gamma$ is well defined since every $\xi_\gamma\in \H_\gamma$ ``appears'' in one and only one family $(\xi_\alpha)\in \D$. By \eqref{eq_ineduality}, $F_X^\gamma$ is bounded. Hence there exists $X_\gamma \in
\BB(\H_\gamma)$ such that
$$F_X^\gamma(\xi_\gamma, \eta_\gamma)= \ip{\xi_\gamma}{X_\gamma\eta_\gamma}_\gamma, \;\forall \xi_\gamma, \eta_\gamma \in \H_\gamma.$$

The fact that $X=\varinjlim X_\gamma$ is easily checked.
\end{proof}

The set $\LBDD{}$ has an obvious structure of vector space. Moreover it carries an involution $X \to X^\dagger$, defined by the equality
$$ B(X^\dagger(\eta_\alpha), (\xi_\alpha))= \overline{B(X(\xi_\alpha), (\eta_\alpha) )}, \quad (\xi_\alpha), (\eta_\alpha) \in \D.$$
It is easily seen that, if $X=\varinjlim X_\gamma$, then $X^\dagger= \varinjlim X_\gamma^*$, where $X_\gamma^*$ denotes the adjoint of $X_\gamma$ in the Hilbert space $\H_\gamma$.

\berem A consequence of the Proposition \ref{prop_operators1} is the
existence, for every $\alpha \in \F$, of an injective linear map
$\Phi_\alpha: \BB(\H_\alpha) \to \LBDD{}$, preserving the
involutions for which there exists $\overline{\alpha}\in\F$ such
that
$$ \Phi_\alpha (X_\alpha)= \Phi_\beta (U_{\beta\alpha}X_\alpha V_{\alpha\beta}), \; \overline{\alpha}\leq\alpha\leq \beta.$$ The map $\Phi_\alpha$ associates to $X_\alpha \in \BB(\H_\alpha)$ the unique operator $X\in \LBDD{}$ which ``reduces'' to $X_\alpha$ on $\H_\alpha$.
\enrem

Since operators of $\BB(\H_\alpha)$ can be multiplied by each other, it is natural to pose the question if this multiplication may be used to define a (at least, partial) multiplication in $\LBDD{}$ or, in other words, if $\LBDD{}$ can be made into a partial *-algebra in the sense of \cite{pop-book}.

\bedefi In $\LBDD{}$ a partial multiplication $X\cdot Y$ of $X,Y \in \LBDD{}$, with $X= \varinjlim X_\beta$ and $Y= \varinjlim Y_\beta$, is defined by the conditions:
\begin{align*}\label{_defn_multiplication}
&\exists \gamma\in \F:\,\,X_\beta Y_\beta =
U_{\beta\alpha}X_{\alpha}V_{\alpha\beta}U_{\beta\alpha}Y_{\alpha}V_{\alpha\beta}, \; \forall\beta,\alpha\geq\gamma \\
& X\cdot Y= \varinjlim X_\beta Y_\beta
\end{align*}
\findefi
A rather simple check then shows that
\begin{prop}\label{prop_pa}
$\LBDD{}$ is a partial *-algebra with respect to the usual operations and the multiplication defined above.
\end{prop}

\end{document}